\documentclass[12pt,reqno]{amsart}
\UseRawInputEncoding

\usepackage{amsmath,amssymb,amscd,amsfonts}
\usepackage{dsfont}
\usepackage[all]{xy}
\usepackage{enumerate}

\usepackage{graphics}
\usepackage{graphicx}

{\catcode`\@=11
\gdef\n@te#1#2{\leavevmode\vadjust{%
 {\setbox\z@\hbox to\z@{\strut#1}%
  \setbox\z@\hbox{\raise\dp\strutbox\box\z@}\ht\z@=\z@\dp\z@=\z@%
  #2\box\z@}}}
\gdef\leftnote#1{\n@te{\hss#1\quad}{}}
\gdef\rightnote#1{\n@te{\quad\kern-\leftskip#1\hss}{\moveright\hsize}}
\gdef\?{\FN@\qumark}
\gdef\qumark{\ifx\next"\DN@"##1"{\leftnote{\rm##1}}\else
 \DN@{\leftnote{\rm??}}\fi{\rm??}\next@}}

\DeclareFontFamily{OT1}{wncyr}{\hyphenchar\font45 }
\DeclareFontShape{OT1}{wncyr}{m}{n}{%
   <5> <6> <7> <8> <9> gen * wncyr
   <10> <10.95> <12> <14.4> <17.28> <20.74>  <24.88>wncyr10}{}
\DeclareFontShape{OT1}{wncyr}{m}{it}{%
   <5> <6> <7> <8> <9> gen * wncyi
   <10> <10.95> <12> <14.4> <17.28> <20.74> <24.88> wncyi10}{}
\DeclareFontShape{OT1}{wncyr}{m}{sc}{%
   <5> <6> <7> <8> <9> <10> <10.95> <12> <14.4>
   <17.28> <20.74> <24.88>wncysc10}{}
\DeclareFontShape{OT1}{wncyr}{b}{n}{%
   <5> <6> <7> <8> <9> gen * wncyb
   <10> <10.95> <12> <14.4> <17.28> <20.74> <24.88>wncyb10}{}
\input cyracc.def

\DeclareMathSizes{9}{9}{7}{5}

\DeclareMathSymbol{\twoheadrightarrow} {\mathrel}{AMSa}{"10}

\theoremstyle{plain}

\newtheorem{theorem}{Theorem}

\newtheorem{lemma}{Lemma}
\newtheorem{corollary}{Corollary}

\theoremstyle{definition}

\newtheorem{definition}{Definition}

\theoremstyle{remark}

\newtheorem*{rem}{\it Remark}

\def\s3{{\rm SL}_3(\mathbb C)}

\begin{document}

\title[The variety of flexes of plane cubics]
{The variety of flexes of plane cubics}

\author[Vladimir L. Popov]{Vladimir L. Popov${}^{1}$}
\thanks{
${}^1$  Steklov Mathematical Institute,
Russian Academy of Sciences, Gub\-kina 8, Mos\-cow
119991, Russia, {\rm popovvl@mi-ras.ru}.
\\[.5mm]
\indent
This work was supported by the Russian Science Foundation under grant no. 23-11-00033, {\rm https:/\hskip -1mm/rscf.ru/en/project/23-11-00033/}
}


\vskip 2mm

\begin{abstract}
Let $X$ be the variety of flexes of plane cubics. We prove that
(1) $X$  is an irreducible  rational algebraic variety endowed with a faithful algebraic action of
${\rm PSL}_3$;
(2)  $X$ is ${\rm PSL}_3$-equiva\-riant\-ly birati\-o\-nally
isomorphic to a homoge\-neous fiber space over ${\rm PSL}_3/K$
with fiber $\mathbb P^1$ for some subgroup $K$ isomorphic to
the binary tetra\-hed\-ral group ${\rm SL}_2(\mathbb F_3)$.
\end{abstract}

\maketitle

\section{Introduction}

We consider
the $3$-dimensional complex coordinate vector space
\begin{equation*}
V:=\mathbb C^3.
\end{equation*}
Let
$x_0, x_1, x_2\in V^*$ be the standard coordinate functions on $V$. In the $10$-dimensional complex vector space
\begin{equation*}
U:={\sf S}^3(V^*)
\end{equation*}
of degree $3$ forms on $V$, all monomials $x_0^{i_0}x_1^{i_1}x_2^{i_2}$ with $i_0+i_1+i_2=3$, ordered in some way, form
a basis.\;Let $\{\alpha_{i_0 i_1 i_2}\mid i_0+i_1+i_2=3\}$ be the dual basis of
$U^*$.

The sets of forms  $\{x_j\}$ and $\{\alpha_{i_0 i_1 i_2}\}$ are the projective coordinate systems on the  projective spaces
${\mathbb P}(V)$ and ${\mathbb P}(U)$
of one-dimensional linear subspaces of
 $V$ and $U$
 respectively.
 Let
\begin{equation}\label{pU}
{p}_U\colon
U\setminus \{0\}\to {\mathbb P}(U)
\end{equation}
be the canonical projection.

We consider the following forms on
 $ {\mathbb P}(U)\times {\mathbb P}(V)
 $:
\begin{align}\label{FH}
F&:=\sum_{i_0+i_1+ i_2=3} \alpha_{i_0 i_1 i_2} x_0^{i_0}x_1^{i_1}x_2^{i_2},\\
\label{H}
H&:=
{\rm det}\Big(\frac{\partial^2F}{\partial x_i\partial x_j}\Big).
\end{align}
They determine the closed subset
\begin{equation}\label{X}
X:=\{a\in
{\mathbb P}(U)\times {\mathbb P}(V)
\mid F(a)=H(a)=0\}
\end{equation}
of ${\mathbb P}(U)
\times {\mathbb P}(V).
$
Let
\begin{equation}\label{pp}
{\mathbb P}(V)\overset{\hskip 1.5mm {\pi}_2}{\longleftarrow} X\overset{{\pi }_9}{\longrightarrow} {\mathbb P}(U)
\end{equation}
be the natural projections.
If
$f\in U$ is a nonzero form such that the cubic
\begin{equation}\label{cubic}
C(f):=\{c\in
{\mathbb P}(V)\mid f(c)=0\}
\end{equation} is an elliptic curve, then
${\pi}_2\big({\rm \pi}_9^{-1}(p_U (f))\big)$
is
the set of all  flexes  (i.e., inflection points) of
$C(f)$, see
\cite[pp.\,293--294]{BK}.
This
entails that a dense subset
of
$X$
is identified with the set of all pairs
 $(C, c)$, where $C$ is an elliptic curve in ${\mathbb P}(V)
 $, and $c$ is its flex (in fact, in Lemma \ref{S}(b) below is shown that this subset is open in $X$, and therefore, is endowed with the structure of an algebraic variety for which $X$ is a compactification).
 For this reason, $X$ is called
  the {\it variety of flexes
 of plane cubics}. The monodromy of $\pi_9$ and the cohomological properties of
 $X$ were explored in
 \cite{Har}, \cite{Kul1}, \cite{Kul2}, \cite{Pop3}. The aim of this paper is to obtain several results about
 other properties of
 $X$. Namely, we prove the following.

 The complex algebraic group
\begin{equation*}
G:=\s3
\end{equation*}
naturally acts on $V$, $U$, ${\mathbb P}(V)$, ${\mathbb P}(U)$, ${\mathbb P}(U)\times {\mathbb P}(V)$. The set $X$ is $G$-stable. The inefficiency kernel of the $G$-actions on ${\mathbb P}(V)$, ${\mathbb P}(U)$, ${\mathbb P}(U)\times {\mathbb P}(V)$, $X$ is the center $Z$ of $G$, therefore, these actions determine the faithful actions of the quotient group
\begin{equation*}
PG:=G/Z={\rm PSL}_3(\mathbb C)
\end{equation*}
  on these varieties.   Let $\varepsilon\in\mathbb C$ be a primitive cubic root of $1$. For the Galois field $\mathbb F_3:=\mathbb Z/3\mathbb Z$, denote
  ${\boldsymbol z}
 :=z+3\mathbb Z\in
 \mathbb F_3$. Let $\mathcal F$  be the set of the following nine points of ${\mathbb P}(V)$ numbered by the elements of $\mathbb F_3^2$ (the reason for choosing such numbering will become clear from what follows, see Lemma \ref{proC}$({\rm H}{}_{10})$,$({\rm H}{}_{11})$):
 \begin{equation}\label{pij}
\left.
\begin{split}
t_{\boldsymbol{0},\boldsymbol{0}}&:=(0:-1:1), &\; t_{\boldsymbol{0},\boldsymbol{1}}&:=(0:-\varepsilon:1),&\; t_{\boldsymbol{0},\boldsymbol{2}}&:=(0:-\varepsilon^2:1),\\
t_{\boldsymbol{1},\boldsymbol{0}}&:=(1:0:-1), &\; t_{\boldsymbol{1},\boldsymbol{1}}&:=(1:0:-\varepsilon),&\; t_{\boldsymbol{1},\boldsymbol{2}}&:=(1:0:-\varepsilon^2),\\
t_{\boldsymbol{2},\boldsymbol{0}}&:=(-1:1:0), &\; t_{\boldsymbol{2},\boldsymbol{1}}&:=(-\varepsilon:1:0),&\; t_{\boldsymbol{2},\boldsymbol{2}}&:=(-\varepsilon^2:1:0).
\end{split}
\right\}
\end{equation}
 Then the $PG
$-normalizer of $\mathcal F$, i.e.,
$$
N_{PG, \mathcal F}:=\{g\in PG
\mid g\cdot \mathcal F=\mathcal F\},
$$
is the so-called Hessian group $\rm Hes$ of order $216$ (see below
its  definition \eqref{fHes}
and Lem\-ma\;\ref{proC}$({\rm H}{}_9)$).
It acts transitively on $\mathcal F$. For any $(\boldsymbol{i},\boldsymbol{ j})\in\mathbb F_3^2$, the stabi\-lizer of $t_{\boldsymbol{i},\boldsymbol{j}}$ with respect to this action is a subgroup ${\rm Hes}_{\boldsymbol{i},\boldsymbol{j}}$ of $\rm Hes$  of order $24$. It is isomorphic to the binary tetrahedral group ${\rm SL}_2(\mathbb F_3)$ and naturally acts on the projective line $\ell$ parameterizing the Hesse pencil
of cubics in ${\mathbb P}(V)$ with the set of flexes $\mathcal F$. Let
$PG
\times^{{\rm Hes}_{\boldsymbol{i},\boldsymbol{j}}} \ell$
be the
homogeneous fiber space
over $PG
/{\rm Hes}_{\boldsymbol{i},\boldsymbol{j}}$ with fiber $\ell$ determined by this action (see its definition in Subsection  \ref{hfs}).

 The  following Theo\-rems \ref{main1}--\ref{main3}
  are the main results of this paper.
  \begin{theorem}\label{main1} The algebraic variety
 $X$ is irreducible.
 \end{theorem}

 The proof of Theorem \ref{main1} is given in Subsection \ref{pr1}.

  \begin{theorem}\label{main2}\

  \begin{enumerate}[\hskip 4.2mm\rm(a)]
  \item The algebraic varieties
  $X$ and $PG  \times^{{\rm Hes}_{\boldsymbol{i},\boldsymbol{j}}} \ell$ are $PG$-equiva\-riantly birationally iso\-morphic.
  \item  The homogeneous fiber space
$PG  \times^{{\rm Hes}_{\boldsymbol{i},\boldsymbol{j}}} \ell$
over $PG/{\rm Hes}_{\boldsymbol{i},\boldsymbol{j}}$ with fiber $\ell$ is the
  projectivization of a homogeneous vector bundle
  over $PG/{\rm Hes}_{\boldsymbol{i},\boldsymbol{j}}$ of rank $2$.
  \end{enumerate}
 \end{theorem}

 The proof of Theorem \ref{main2} is given in Subsection \ref{pr2}.

The next  Theorem \ref{main3} was announced in \cite{Pop3}, where unirationa\-lity of $X$ was proved.

 \begin{theorem}\label{main3}
 The algebraic variety
 $X$ is rational.
  \end{theorem}

 The proof of Theorem \ref{main3} is given in Subsection \ref{pr3}.

 \vskip 3mm

 \noindent {\it Conventions and notation}

 \vskip 1mm

 We use the standard notation and terminology from
 \cite{Bor}, \cite{Sha}, \cite{PV}.

 Unless otherwise stated, all topological terms refer to the Zariski topo\-logy. If $S$ is a subset of an algebraic variety $Y$, then $\overline S$ is the closure of $S$ in $Y$ (whenever this notation is used, it is either explicitly specified or clear from the context what $Y$ for $\overline S$ is meant).

 Given an action $\alpha\colon R\times A\to A$ of a group $R$ on a set $A$ and the elements
 $r\in R$, $a\in A$, we denote $\alpha(r, a)$ by $r\cdot a$. Given a subset $B$ of $A$, the $R$-normalizer and $R$-centralizer of $B$ are respectively the following subgroups of $R$:
 \begin{equation}\label{NZ}
 \begin{split}
 N_{R, B}&:=\{r\in R\mid r\cdot B\subseteq B\},\\
 Z_{R, B}&:=\{r\in R\mid r\cdot b=b \;\mbox{ for every $b\in B$}\}.
 \end{split}
 \end{equation}

If $R$ is an algebraic group, $A$ is an algebraic variety, and $\alpha$ is a morphism, then $Z_{R, B}$ is closed and,
provided
$B$ is closed,
$N_{R, B}$ is closed, too (see \cite[Chap.\,I, 1.7]{Bor}).

$\mathbb C^*$ is the multiplicative group of $\mathbb C$.

${\rm Aff}(A)$ is the group of all affine transformations of a finite-dimensional affine space $A$ over some field.

${\rm SAff}(A)$ is the normal subgroup of ${\rm Aff}(A)$ consisting of all elements whose linear part has determinant $1$.

 \section{Irreducibility of $X$}

 \subsection{\it Some group actions}
  Apart from the actions of $G$ on $V$, $U$, ${\mathbb P}(V)$, ${\mathbb P}(U)$, below we
consider
the action of
$\mathbb C^*$ on $U$ by scalar multiplica\-tion. This action commutes with that of $G$.

The following Lemma \ref{class} summarizes some known
facts about $G$-orbits and their closures in $U$
that we need.

\begin{lemma}[{{\rm \cite{Poi}, \cite[Chap.\,1, \S7]{Kra}}}]
\label{class} Let $f\in U$ be a nonzero form.
\begin{enumerate}[\hskip 4.2mm\rm(a)]
\item
$C(f)$
is not an elliptic curve
if and only if the orbit $G\cdot f$ contains a form $h$ from the following Table $1$:
\end{enumerate}

\vskip 3mm
\centerline{\sc Table $1$}

\vskip 1mm
\begin{center}
\begin{tabular}{c||c|c}
$h$&$\dim G\cdot h$&$C(h)$
\\[2pt]
 \hline
 \hline
  $h_3:=x_0^3$&$3$&{\rm \footnotesize line}\\
  \hline
  $h_5:=x_0^2x_1$&$5$&$\begin{smallmatrix}
  \mbox{{\rm \footnotesize two lines}}
  \end{smallmatrix}$\\
 \hline
$h_6:=x_0x_1(x_0+x_1)$&$6$&
$\begin{smallmatrix}
\\
  \mbox{{\rm \footnotesize three lines}}\\
  \mbox{{\rm \footnotesize passing through}}\\
  \mbox{{\rm \footnotesize a point}}\\
  \\[-1mm]
  \end{smallmatrix}$
\\
 \hline
 $h_{\mu, 6}:=\mu x_0x_1x_2$,\;$\mu\in \mathbb C^*$&$6$&
 $\begin{smallmatrix}
 \\
  \mbox{{\rm \footnotesize three lines}}\\
  \mbox{{\rm \footnotesize with empty}}\\
  \mbox{{\rm \footnotesize intersection}}\\
  \\[-1mm]
  \end{smallmatrix}$
  \\
 \hline
$h_7:=(x_0^2-x_1x_2)x_1$&$7$&
$\begin{smallmatrix}
\\
  \mbox{{\rm \footnotesize conic and }}\\
  \mbox{{\rm \footnotesize its tangent line}}\\
  \\[-1mm]
  \end{smallmatrix}$
\\
 \hline
$h_{\mu, 7}:=\mu (x_0^2-x_1x_2)x_0$,\;$\mu\in \mathbb C^*$&$7$&
$\begin{smallmatrix}
\\
  \mbox{{\rm \footnotesize conic and }}\\
  \mbox{{\rm \footnotesize nontangent line}}\\
  \\[-1mm]
  \end{smallmatrix}$
  \\
 \hline
$h_8:=x_1^2x_2-x_0^3$&$8$&
$\begin{smallmatrix}
  \mbox{{\rm \footnotesize cuspidal cubic}}
  \end{smallmatrix}$
  \\
 \hline
$h_{\mu,8}:=\mu (x_1^2x_2-x_0^3-x_0^2x_2)$,\;$\mu\in \mathbb C^*$&$8$&
$\begin{smallmatrix}
  \mbox{{\rm \footnotesize nodal cubic}}
  \end{smallmatrix}$\\
\end{tabular}
\end{center}

\vskip 4mm

\begin{enumerate}[\hskip 4.2mm\rm(b)]
\item[\rm(b)]
If $C(f)$
is not an elliptic curve, then such $h$ is uniquely deter\-mined by $f$.

\item[\rm(c)]
All the pairs of $G$-orbits $\mathcal O_1$, $\mathcal O_2$ in $U$ such that $\overline{\mathcal O_1}\supset \mathcal O_2$ and $\mathcal O_1\neq  \mathcal O_2$
are described by the following inclusions:
\begin{gather*}
\overline{G\cdot h_{\mu, 8}}\supset \overline{G\cdot h_{2 \mu, 7}} \supset G\cdot h_{2 \mu, 6},\\
\overline{G\cdot h_{8}}\supset \overline{G\cdot h_{7}} \supset \overline{G\cdot h_{6}}
\supset \overline{G\cdot h_{5}}\supset \overline{G\cdot h_{3}}\ni 0.
\end{gather*}

\item[\rm(d)]
If
$C(f)$ is
an elliptic curve,
then $\dim G\cdot f=8$.

\item[\rm(e)]
$\dim G\cdot f<8$ if and only if
$f$ is a reducible form.
\end{enumerate}
\end{lemma}

In view of this lemma,
\begin{equation*}
\underset{f\in U}{\max}\dim G\cdot f=8.
\end{equation*}
By \cite[1.4]{PV}, this and $\dim G=8$ yield
\begin{corollary}\label{dim} The $G$-stable and $\mathbb C^*$-stable set
\begin{equation}\label{pos}
\{f\in U \mid
\mbox{\rm the $G$-stabilizer of $f$ is finite}\}
\end{equation}
is nonempty
and open
in $U$.
\end{corollary}

If $f\in U$, let $H_f$ be the form on $V$ obtained from $H$ by replacing every $\alpha_{i_1i_2i_3}$
with $\alpha_{i_1i_2i_3}(f)\in\mathbb C$ in the right-hand side of \eqref{H}.

For $f\neq 0$, we put
\begin{equation}\label{Ff}
{\rm Fl}\big(C(f)\big):=C(f)\cap \{c\in {\mathbb P}(V)\mid H_f(c)=0\}.
\end{equation}
Since $C(f)$ is an algebraic curve,
\begin{equation}\label{<2}
\dim {\rm Fl}\big(C(f)\big)\leqslant 1.
\end{equation}

In view of the known property of Hessian \cite[3.3.13]{Spr},  we have
$H_{g\cdot f}(g\cdot v)=H_f(v)$ for every $g\in G$, $f\in U$, $v\in V$.
Therefore,
\begin{equation}\label{gFl}
{\rm Fl}\big(C({g\cdot f})\big)=g\cdot {\rm Fl}\big(C(f)\big).
\end{equation}

From \eqref{gFl} it follows that $X$ is a $G$-stable subset of ${\mathbb P}(U)\times {\mathbb P}(V)$.

\begin{lemma}\label{Fl1} For every nonzero form $f\in U$, the following
are
equivalent:
\begin{enumerate}[\hskip 4.2mm\rm(a)]
\item $\dim {\rm Fl}\big(C(f)\big)=1$;
\item the $G$-stabilizer of $f$ is infinite;
\item the orbit $G\cdot f$ contains one {\rm(}and only one{\rm)} of the forms
$$h_{\mu, 7}, h_7, h_{\mu, 6}, h_6, h_5, h_3;$$
\item $f$ is reducible.
\end{enumerate}
\begin{proof} Let $f$ be a form $h$ from Table $1$. Then the following Table 2 holds, in which $c$ denotes the number of irreducible components of the curve $C(f)$ lying in ${\rm Fl}(C(f))$:

\eject

\centerline{\sc Table $2$}

\vskip 3mm
\begin{center}
\begin{tabular}{c||c|c|c}
$f$&$H_f$&$\dim {\rm Fl}\big(C(f)\big)$&$c$
\\[2pt]
\hline
 \hline
 $h_3$&$0$&$1$&$1$ \\
 \hline
$ h_5$&$0$&$1$&$2$
 \\
 \hline
 $h_6$&$0$&$1$&$3$
 \\
 \hline
 $h_{\mu,6}$& $2\mu^3x_0x_1x_2$&$1$&$3$
 \\
 \hline
 $h_7$&$-8x_1^3$&$1$&$1$
 \\
 \hline
 $h_{\mu,7}$& $-2\mu^3(3x_0^2+x_1x_2)x_0$&$1$&$1$
 \\
 \hline
 $h_8$&$24x_0x_1^2$&$0$&$0$
 \\
 \hline
 $h_{\mu,8}$&$8{\mu}^3(-x_0^2x_2+3x_0x_1^2+x_1^2x_2)$&$0$&$0$
 \end{tabular}
\end{center}

\vskip 1.5mm

\noindent Indeed, the second column of Table 2 is obtained from the first column of Table 1 and formulas \eqref{FH}, \eqref{H} by a straightforward simple computa\-tion. The third column of Table 2 is deduced from formulas \eqref{cubic}, \eqref{Ff} by analyzing solutions of a simple system of two polynomial equations in three variables.

If $C(f)$ is an elliptic curve,
then ${\rm Fl}\big(C(f)\big)$ is the set of all flexes of $C(f)$, hence a finite set of cardinality $9$, \cite[p.\,291, Cor.\,3]{BK}; whence  $\dim {\rm Fl}\big(C(f)\big)=0$. Now the claim follows from \eqref{gFl}, Lemma 1(a),(d),(e), and comparing the last columns of Tables\,1 and\,2.
\end{proof}
\end{lemma}

\begin{lemma}\label{S} \

\begin{enumerate}[\hskip 4.2mm\rm (a)]
\item\label{Sa} The set
\begin{equation*}
J:=\{f\in U\setminus \{0\}\mid \dim {\rm Fl}\big(C(f)\big)=1\}
\end{equation*}
is $G$-stable, $\mathbb C^*$-stable, closed in $U\setminus \{0\}$,
and $8$-dimensional.
\item\label{Sb} There is an irreducible homogeneous
$G$-invariant polynomial $\Delta\in \mathbb C[U]$ such that for every form $f\in U$,
the following properties are equivalent:
    \begin{enumerate}[\hskip 0mm\rm(i)]
    \item $\Delta(f)\neq 0$,
        \item $C(f)$ is an elliptic curve.
    \end{enumerate}
    \end{enumerate}
\end{lemma}
\begin{proof}
\eqref{Sa} Lemma \ref{Fl1} entails that
 \begin{equation}\label{Sc}
 \mbox{$J$ is the complement in $U\setminus\{0\}$ to the set \eqref{pos}.}
  \end{equation}
  By Corollary \ref{dim}, this proves
    the claims about $G$-stability, $\mathbb C^*$-stability, and closedness of $J$. It remains to
    prove that $\dim J=8$.

 We consider the categorical quotient for the action of $G$ on $U$:
\begin{equation*}\label{cq}
\pi_U\colon
U\rightarrow
U/\!\!/G.
\end{equation*}
By the general properties of such quotients  (see \cite[4.4]{PV}),
$\pi_U$ is surjec\-tive,
every fiber of $\pi_U$ is a closed $G$-stable subset of $U$ containing a unique closed $G$-orbit, and
the $\mathbb C^*$-action on $U$
descends to a $\mathbb C^*$-action on $U/\!\!/G$ such that $\pi_U$ is $\mathbb C^*$-equivariant.
From this, Lemma \ref{class}(a),(c),(d), and the equality $\dim G=8$, we infer that
every fiber $E$ of
$\pi_U$ shares the following properties:
\begin{enumerate}[\hskip 3.2mm \rm(a)
]
\item[$({\rm q}_1)$]
$E
$ is irreducible and $\dim E
=8$.
\item[$({\rm q}_2)$]
$E
$ contains an open $G$-orbit $\mathcal O_E$.
\item[$({\rm q}_3)$]
$E
\setminus \mathcal O_E\!\neq\!\varnothing$ if and only if $h_{\mu, 8}$ or $h_8\in \mathcal O_E$.\;If
$E\setminus \mathcal O_E\neq \varnothing$,
then  $E
\setminus \mathcal O_E$ is the closed irreducible $7$-dimensional set
$\overline{G(h_{2\mu, 7})}$ or, respectively,
$\overline{G(h_7)}$ whose unique closed $G$-orbit is  $G(h_{2\mu, 6})$ or, respectively, $\{0\}$.
\end{enumerate}

Let
$Y$ be the closure in $U/\!\!/G$ of the $\mathbb C^*$-orbit
of $\pi_U(h_{1,6})$. It is an irreducible algebraic curve. The fixed point $\pi_U(0)$ of the $\mathbb C^*$-action on $U/\!\!/G$
lies in $Y$ because $0\in \overline {\mathbb C^*\cdot h_{1,6}}$. Since every orbit closure of any reductive group action on an affine variety contains a unique closed orbit, this means that $\pi_U(0)$ is the complement  in $Y$ to the $\mathbb C^*$-orbit
of $\pi_U(h_{1,6})$.

We consider in $U$ the closed $G$-stable and $\mathbb C^*$-stable set $\pi_U^{-1}(Y)$.
 Lemma \ref{class}(c), the definition of $Y$, and property
 $({\rm q}_1)$ entail that the set
$\pi_U^{-1}(Y)\setminus\{0\}$ is the union
of $G$-orbits of all forms $f$ from Table 2. From this and \eqref{Sc} we infer that $J$ is
a closed subset of $\pi_U^{-1}(Y)\setminus\{0\}$, the morphism $\pi_U\colon J\to Y$ is surjective, and every its fiber is $7$-dimensional. Therefore, $\dim J=\dim Y+7=8$.

\vskip 1mm

\eqref{Sb}
The variety $U/\!\!/G$ is isomorphic to the affine plane $\mathbb A^2$ (see \cite[0.14]{PV}). Therefore, the ideal of $Y$ in $\mathbb C[U/\!\!/G]$ is principal. Let $\delta$ be its generator. Since the set $Y$ is irreducible and $\mathbb C^*$-stable,   $\delta$ is an irredu\-cible and $\mathbb C^*$-semi-invariant element.
Hence
$\Delta:=\pi_U^*(\delta)$
is an irreducible and $\mathbb C^*$-semi-invariant (i.e., homogeneous) element of  $\mathbb C[U]^G$. Since the group $G$ is
connected and has no nontrivial characters, $\Delta$ is also an irreducible element of $\mathbb C[U]$ (see \cite[Lem.\,2]{Pop1}).
Finally, $\pi_U^{-1}(Y)=\{f\in U\mid \Delta(f)=0\}$ and, as explained above, $\pi_U^{-1}(Y)\setminus \{0\}$ is the union of $G$-orbits of all forms $h$ from Table 1. Now the equivalence (i)$\Leftrightarrow$(ii) follows from
Lemma \ref{class}(a).
\end{proof}

\begin{corollary}\label{pS}
$p_U(J)$ is closed in ${\mathbb P}(U)$ and $\dim p_U(J)=7$.
\end{corollary}
\begin{proof} If a subset $K$ of $U\setminus\{0\}$ is $\mathbb C^*$-stable and closed in
$U\setminus\{0\}$, then $p_U(K)$ is closed in ${\mathbb P}(U)$ and $\dim K=\dim p_U(K)+1$.
Therefore, the claim follows from Lemma \ref{S}.
\end{proof}

\subsection{\it Proof of Theorem {\rm\ref{main1}}}\label{pr1} For every form $f\in U\setminus \{0\}$, the restriction of $\pi_2$
to the
fiber $\pi_9^{-1}(p_U(f))$
is a bijection
 to
${\rm Fl}\big(C(f)\big)$ (see \eqref{pU}, \eqref{pp}).
Hence, by \eqref{<2} and Lemma\,\ref{S}, we have
\begin{equation}\label{ddfi}
\dim \pi_9^{-1}\big(p_U(f)\big)=
\begin{cases}
1& \mbox{if $f\in J$,}\\
0&\mbox{if $f\notin J$.}
\end{cases}
\end{equation}

Since $X$ is cut off in the $11$-dimensional variety ${\mathbb P}(U)\times {\mathbb P}(V)$ by two equations
(see \eqref{X}),
every irreducible component of $X$ is either $9$- or $10$-dimensional (see \cite[Cor.\,1.7]{Sha}).

First,
$X$ has no irreducible components of dimension 10. Indeed, arguing by contradiction, we assume that $X'$ is such a component. 
Since $\dim {\mathbb P}(U)=9$ and every fiber of $\pi_9$ is at most $1$-dimensional, this implies that all fibers of $\pi_9\colon X'\to {\mathbb P}(U)$ are $1$-dimensional. On the one hand, the latter entails that $\pi_9(X')={\mathbb P}(U)$, so $\dim \pi_9(X')=9$.
But on the other hand, in view of \eqref{ddfi}, it entails that $p_9(X')\subseteq p_U(J)$,
whence $\dim \pi_9(X')\leqslant 7$ by Corollary \ref{pS}. A contradiction.

Thus, every irreducible component $X'$ of $X$ is $9$-dimensional. Since every fiber of $\pi_9$ is at most $1$-dimensional, we have $9\geqslant \dim \pi_9(X')\geqslant 8$, which entails by Corollary \ref{pS} that $\pi_9(X') \nsubseteq p_U(J)$. Therefore,
\begin{equation}
\label{sur}
\pi_9(X')={\mathbb P}(U)
\end{equation} in view of  \eqref{ddfi}.

Again arguing by contradiction, we assume that $X$ has more than one irreducible component. In view of \eqref{sur}, there is a Zariski open subset ${\mathbb P}(U)^0$ of ${\mathbb P}(U)$ such that $X^0:=\pi_9^{-1}({\mathbb P}(U)^0
)$ is smooth, intersects every irreducible component of $X$, and
\begin{equation}\label{Theta}
\pi_9\colon
X^0\to {\mathbb P}(U)^0
\end{equation}
is an unbranched covering whose restriction to every irreducible compo\-nent of $X^0$ is surjective. Then $X^0 $ has more than one  irreducible compo\-nent and these components are pairwise disjoint.
We take a point $a\in {\mathbb P}(U)^0$. The finite set $\pi_9^{-1}(a)$ intersect every irreducible component of $X^0$.
By \cite[II, 2, Prop.]{Har}, the monodromy group of
\eqref {Theta} acts transitively on the finite set $\pi_9^{-1}(a)$. Therefore, there are two points of $\pi_9^{-1}(a)$ that lie in different irreducible components of $X^0$  and are
connected by a continuous (in strong $\mathbb C$-topology) path in $X^0$. This contradicts the property that these components are disjoint.\hfill $\Box$

\section{Relative section for the action of $PG$ on $X$}\label{rela}

We
start
with considering
in Subsections \ref{hfs}, \ref{rs}
two general const\-ructions  from the theory of algebraic trans\-for\-mation groups that we will need.

\subsection{\it Homogeneous fiber spaces {\rm (\cite[3.2]{Ser}, \cite[4.8]{PV})}}\label{hfs}
Let $R$ be a con\-nected algebraic group, let $Q$ be a closed subgroup of $R$, and let
$S$ be an algebraic variety endowed with an algebraic
action of $Q$. Then $Q$ acts on $R\times S$ by the formula
\begin{equation}\label{actD}
q\cdot (r, s)\mapsto (rq^{-1}, q\cdot s).
\end{equation}

A mild restriction on $S$ ensures the existence of a quotient for this action (in the sense of \cite[6.3]{Bor}): such a quotient exists if
\begin{equation}\label{*}
\begin{split}
\mbox{every finite subset
of $S$
lies
in an affine open subset
of $S$.}
\end{split}
\end{equation}
For instance, every quasiprojective $S$ shares property\eqref{*}. Hence if $S$ shares it, then every locally closed subset of $S$ shares it as well.

The specified quotient is denoted by
 \begin{equation}\label{quo}
\delta_{R,  Q, S}\colon R\times S\to R\times^Q S.
\end{equation}
The natural projection $R\times S\to R$ is $Q$-equivariant with respect to the actions of $Q$ on $R$ and $R\times S$ respectively by right multiplication and
through the first factor.
Therefore, by the universal mapping property of quotients \cite[6.3]{Bor}, it
induces the surjective morphism  of the quotient varieties
\begin{equation*}
\pi_{R, Q, S}^{\ }\colon R\times^Q S\to R/Q.
\end{equation*}
Every fiber of $\pi_{R, Q, S}$ is isomorphic to $S$. As the $R$-action on $R\times S$ by left multiplication of the first factor
commutes with
$Q$-action \eqref{actD}, the former  action descends to an $R$-action on $R\times^Q S$; the morphism $\pi_{R, Q, S}^{\ }$ is equivariant with respect to this action and the natural $R$-action on $R/Q$. Given the aforesaid,
$R\times^QS$ is called the {\it $($algebraic$)$ homogeneous fiber space over $R/Q$ with fiber $S$}.

In general,
the fibration
$\pi_{R, Q, S}^{\ }$
over $R/Q$ with fiber $S$
is locally trivial in the \'etale topology, but not in
the Zariski topology, i.e., every point of  $R/Q$ has a neighborhood, over a suitable
unramified covering of which, the induced fibration is trivial with fiber $S$. However, if $S$ is a vector space
over $\mathbb C$ and the $Q$-action on $S$
is linear, then  this fibration is locally trivial in the Zariski topology \cite[5.4 and Thm.\;2]{Ser}. In this case,
$R\times^Q S
$  is called the {\it homogeneous vector bundle over $R/Q$ with fiber $S$
}, and the integer $\dim S
$ is called the {\it rank of this bundle}.

\subsection{\it Relative sections
}\label{rs}
Let $M$ be an irre\-ducible al\-geb\-raic variety en\-dowed with an algebraic action of an algeb\-raic group $R$. Let $S$ be a  closed subset of $M$ and let $S_1,\ldots, S_d$ be all its irreducible components.
Since $S$ is closed,
its $R$-normalizer
$
N_{R, S}
$
(see \eqref{NZ}) is a closed subgroup of $R$ whose elements permute the sets
$S_1,\ldots, S_d$.
If a quotient \eqref{quo} for $Q=N_{R, S}$ exists (which happens, for instance, if
$S$ shares property \eqref{*}), then the universal mapping property of quotient  entails
that the morphism
\begin{equation*}
\alpha\colon R\times S\to M, \quad (r, s)\mapsto r\cdot s
\end{equation*}
is included in the following commutative diagram of $R$-equivariant mor\-phisms:
\begin{equation}\label{isors}
\begin{matrix}
\xymatrix@C=5mm@R=5mm{&R\times S\ar[ld]_{\delta_{R, N_{R, S}, S}}\ar[rd]^{\alpha}&\\
R\times^{N_{R, S}} S\ar[rr]^{\hskip 5mm \iota}&& M
}
\end{matrix}.
\end{equation}

\vskip 1mm

Below in this subsection, we use
the introduced notations.

Definition \ref{relsec} and Lemma \ref{rese} below fix the inaccuracy in definition on p.\,160  and Prop. 2.9 of \cite[2.8]{PV} and generalize definition on p.\,24 and Prop. 1.2 of  \cite{Kat}.

\begin{definition}\label{relsec}
A closed subset $S$ of $M$
is called a {\it relative section for the action of $R$ on $M$} if  the following hold:
\begin{enumerate}[\hskip 3.0mm \rm(i)]
\item[$({\rm s}_1)$] $\overline{R\cdot S_i}=M$ for every $i=1,\ldots, d$;
\item[$({\rm s}_2)$] there is a dense open subset $S^0$ of $S$  such that
every $r\in R$ for which
$r\cdot S^0\cap S^0\neq \varnothing$
lies in
$N_{R, S}$.
\end{enumerate}
\end{definition}

Below in this subsection we assume that $S$ is a relative section for the action of $R$ on $M$.

We recall from \cite[p.\,127]{PV} that saying a certain property holds for points $y$ in general position in an algebraic variety $Y$ means it  holds for every point $y$ of a dense open subset of  $Y$ (depending on the property under consideration).

\begin{lemma}\label{rese}
\

\begin{enumerate}[\hskip 4.2mm \rm(a)]
\item The natural action of $N_{R, S}$ on the set of irreducible components of $S$ is transitive.
\item Every $S_i$
is
a relative section for the same action of $R$ on $M$.
\item $N_{R, S_{i}}=N_{N_{R, S}, S_{i}}$.
\end{enumerate}

If a quotient \eqref{quo} for $Q=N_{R, S}$ exists, then
\begin{enumerate}[\hskip 4.2mm \rm(a)]
\item[\rm(d)] $R\times^{N_{R, S}} S$ is irreducible,
\item[\rm(e)] $\iota$ in {\rm \eqref{isors}}
is a birational isomorphism.
\end{enumerate}
\end{lemma}

\begin{proof}  Replacing $S^0$ in Definition \ref{relsec} by $\bigcup_{r\in {N_{R, S}}} r\cdot S^0$, we may (and shall) assume that $S^0$ is $N_{R, S}$-stable. Moreover, since $I:=\bigcup_{i\neq j}S_i\cap S_j$
is a closed $N_{R, S}$-stable subset of $S$ such that $S_i \nsubseteq I$ for every $i$, replacing $S^0$ by  $S^0\setminus I$, we also may (and shall) assume that the dense open subsets $S_i^0:=S^0\cap S_i$  of $S_i$ share the property
\begin{equation}\label{SiSj}
\begin{split}
S^0_l\cap S^0_k=\varnothing\quad \mbox{for all $l\neq k$}.
\end{split}
\end{equation}
We note that $S_1^0,\ldots, S_d^0$ are all irreducible components of
$S^0$,
so the elements of $N_{R, S}$ permute them. We also note that
\begin{equation}\label{intersect}
N_{Q, S_{i_1}\cup\ldots\cup S_{i_k}}=N_{Q, S_{i_1}^0\cup\ldots\cup S_{i_k}^0}
\end{equation}
for any $i_1,\ldots, i_k$ and any subgroup $Q$ of $R$.

\vskip 1mm

(a) By property $({\rm s}_1)$, for every $S_i, S_j$, there are $r_i, r_j\!\in\! R$ such that
$r_i\!\cdot \!S_i^0\cap r_j\!\cdot\! S_j^0\neq \varnothing$, i.e., for $r:=r_j^{-1}r_i$, we have
\begin{equation}\label{rirj}
r\cdot S_i^0\cap S_j^0\neq \varnothing.
\end{equation}
From \eqref{rirj} and property $({\rm s}_2)$,  we obtain
$r\in N_{R, S}$, hence $r
\cdot S_i^0$ is one of the irredu\-cible components of $S^0$. In view of \eqref{SiSj}, \eqref{rirj}, this entails $r\cdot S_i^0= S_j^0$, which proves (a).

\vskip 1mm

(b), (c)  If $r\in R$ and $r\cdot S_i^0\cap S_i^0\neq\varnothing$, then  property $({\rm s}_2)$ for $S$ yields $r\in N_{R, S}$, and
the same argument as above yields $r\cdot S_i^0=S_i^0$.
Taking into account \eqref{intersect}, this proves (b) and (c).

\vskip 1mm

(d) It follows from (a) that for the $N_{R, S}$-action on $R\times S$ defined by \eqref{actD}, every orbit intersects the closed subset $R\times S_1$. As the fibers of $\delta_{R, N_{R, S}, S}$ are orbits of this action (see \cite[6.3]{Bor}), this means that the restriction of $\delta_{R, N_{R, S}, S}$ to $R\times S_1$ is a surjective morphism to $R\times^{N_{R, S}} S$. This entails irreducibility of  $R\times^{N_{R, S}} S$ because $R\times S_1$ is irreducible.

\vskip 1mm

(e) The morphism $\alpha$ in \eqref{isors} is dominant  in view of property $\rm (s{}_1)$.  Since $\delta_{R, N_{R, S}, S}$ is surjective, this entails that $\iota$ is dominant, too.
As $S^0$ is a dense open $N_{R, S}$-stable subset of $S$, the set $\delta_{R, N_{R, S}, S}(R\times S^0)=R\times^{N_{R, S}} S^0$ is open in $R\times^{N_{R, S}} S$. Therefore,
(f) will be proved if we show that the restriction of $\iota$ to $R\times^{N_{R, S}} S^0$ is injective (here we use that ${\rm char}\,\mathbb C=0$). To show this, take two points
$a_1, a_2\in R\times^{N_{R, S}} S^0$.
For every $i=1, 2$, there are  $r_i\in R$, $s_i\in S^0$ such that $a_i=\delta_{R, N_{R, S}, S}(r_i, s_i)$. Assume that $\iota(a_1)=\iota(a_2)$. This yields, in view of $$\iota(a_i)=\alpha(r_i, s_i)=r_i\cdot s_i,$$
that $r_1\cdot s_1=r_2\cdot s_2$, i.e., $n\cdot s_1=s_2$ where $n:=r_2^{-1}r_1$.  By $\rm (s{}_2)$ this entails $n\in N_{R, S}$ and $$(r_2, s_2)=(r_1n^{-1}, n\cdot s_1);$$ whence $a_1=a_2$.
\end{proof}

Let $\widetilde{M}$ be
an irreducible algebraic variety endowed with an algebraic $R$-action, let $\tau\colon \widetilde{M}\to M$ be an $R$-equivariant dominant mor\-phism, and let $T_1,\ldots, T_k$ be all irreducible components of $\tau^{-1}(S)$.

Definition \ref{defregu} and Lemma \ref{55} below fix the inaccuracy in \cite[Prop.\,2.10]{PV},  \cite[Thm. (1.7.5)]{Pop2}, and generalize \cite[Prop.\,1.2]{Kat}.

\begin{definition}\label{defregu}
An irreducible component $T_i$ is called {\it regular} if
$\overline{\tau(T_i)}$ coincides with
one of $S_1,\ldots, S_d$ and
\begin{equation}\label{ddddd}
\dim \big(T_i\cap \tau^{-1}(m)\big)=\dim \widetilde{M}-\dim M
\end{equation}
for points $m$ in general position in $\overline{\tau(T_i)}$.
\end{definition}

\begin{lemma}\label{55}\

\begin{enumerate}[\hskip 4.2mm\rm (a)]
\item Regular irreducible components of $\tau^{-1}(S)$ exist.
\item
The union $\widetilde{S}$ of all
regular irreducible components of
$\tau^{-1}(S)$
is a relative section
for the action of $R$ on\;$\widetilde{M}$.
\item $N_{R, \widetilde{S}}=N_{R, S}$.
\end{enumerate}
\end{lemma}

\begin{proof}
$\rm (a)$ Since $\tau$ is dominant,
there is a dense open subset $M_0$ of $M$ such that
$\dim \tau^{-1}(m)=\dim \widetilde{M}-\dim M$ for every $m\in M_0$ (see \cite[Thm. 1.25]{Sha}).
We first note that
\begin{equation}\label{M0}
S_j\cap M_0\neq \varnothing\quad \mbox{for every $j$.}
\end{equation}
Indeed, in view of property $({\rm s}_1)$, there are $r\in R, s\in S_j$ such that $r\cdot s\in M_0$. Whence
$\dim \widetilde{M}-\dim M=\dim \tau^{-1}(r\cdot s)
{=} \dim r\cdot \tau^{-1}(s)
{=}\dim \tau^{-1}(s)$ (the second equality
follows from $R$-equivariance of $\tau$, and the third
 from the fact that $R$ acts by automorphisms $M$). Therefore, $s\in S_j\cap M_0$, proving \eqref{M0}.

 Further, in view of (a),
there is an integer $c$ such that
\begin{equation}\label{nnn}
c=\dim S_1=\ldots=\dim S_d.
\end{equation}
The definitions of $S_1,\ldots, S_d$ and $T_1,\ldots, T_k$ yield
\begin{equation}\label{STS}
\left.\begin{split}
S&=S_1\cup\ldots\cup S_d=\tau(\tau^{-1}(S))\\
&=
\tau(T_1\cup\ldots\cup T_k)=
\tau(T_1)\cup\ldots\cup \tau(T_k)\\
&\subseteq
\overline{\tau(T_1)}\cup\ldots\cup \overline{\tau(T_k)}
\subseteq S.
\end{split}
\right\}
\end{equation}
This and \eqref{nnn} show that, in fact, both inclusions in \eqref{STS} are equalities and $c$ is the maximum of the integers $\dim\overline{\tau(T_1)},\ldots, \dim\overline{\tau(T_k)}$. Hence there is $i$ such that
\begin{equation}\label{dn}
\dim\overline{\tau(T_i)}=c.
\end{equation} Since $\overline{\tau(T_i)}$ is irreducible and, by \eqref{STS}, is the union of its intersections with  $S_1,\ldots, S_d$, it coincides with one of these intersections, hence
there is $j$ such that
$\overline{\tau(T_i)}\subseteq S_j$. This and \eqref{nnn}, \eqref{dn} entail that, in fact,
$\overline{\tau(T_i)}=S_j$. Thus, some of the sets $\overline{\tau(T_1)}, \ldots, \overline{\tau(T_k)}$ are irreducible components of $S$. After a suitable re\-num\-bering, we may (and shall) assume that $S_1$ is among these irredu\-cible components and that
$\overline{\tau(T_l)}=S_1$ if and only if $l=1,\ldots, q$. This and \eqref{M0} show that for points $m$ in general position in $S_1$, we have
\begin{equation}\label{top}
\begin{split}
\dim \tau^{-1}(m)&=\dim \widetilde{M}-\dim M,\\
\tau^{-1}(m)&= \big(\tau^{-1}(m)\cap T_1\big)\cup\ldots\cup  \big(\tau^{-1}(m)\cap T_q\big).
\end{split}
\end{equation}
In turn, \eqref{top} entails that there is $i\leqslant q$ such that  $\dim \big(T_i\cap \tau^{-1}(m)\big)=\dim \widetilde{M}-\dim M$. Thus, $T_i$ is regular.
This completes the proof of \rm (a).

\vskip 1mm

(c) The inclusion $N_{R, S}\subseteq N_{R, \widetilde{S}}$ follows from  Definition \ref{defregu} because $\tau$ is $R$-equivariant and $R$ acts by automorphisms of $\widetilde M$.
The reverse inclusion $N_{R, \widetilde{S}}\subseteq N_{R, S}$ follows from the fact that, due to Lemma \ref{rese}(a),  each $S_j$
is equal to $\overline{\tau(T_i)}$ for a suitable $i$.
\vskip 1mm

(b)  Let $T_i$ be a regular irreducible component of $\tau^{-1}(S)$
and let $\overline{\tau(T_i)}=S_j$. As we have $\tau(R\cdot T_i)=R\cdot S_j$ and $\overline{R\cdot S_j}=M$, the morphism
$$\gamma:=\tau|_{\overline{\tau(R\cdot T_i)}}\colon \overline{\tau(R\cdot T_i)}\to M,$$
is dominant. The fibers of $\gamma$ over points
 in general position in $M$ are transformed by $R$ to the fibers of $\gamma$ over points $m$ in general position in $S_j$. As $\gamma^{-1}(m)\supseteq \tau^{-1}(m)\cap T_i$, we then obtain from \eqref{ddddd} that  $\dim \overline{\tau(R\cdot T_i)}\geqslant \dim M+\dim \widetilde{M}-\dim M=\dim\widetilde{M}$. Therefore, $\overline{\tau(R\cdot T_i)}=\widetilde{M}$, i.e., property $\rm(s_1)$ holds for the pair $\widetilde{M}, \widetilde S$.

 Finally, we consider in $\widetilde{S}$ the dense open subset $\widetilde{S}^0:=\widetilde{S}\cap \tau^{-1}(S^0)$. Assume that there are $r\in R$ and $s\in \widetilde{S}^0$ such that $r\cdot s\in \widetilde{S}^0$. Then $S^0\supseteq \tau(\widetilde{S}^0)\ni \tau(r\cdot s)=r\cdot\tau(s)\in r\cdot S^0$. Thus, $r\cdot S^0\cap S^0\neq \varnothing$, hence $r\in N_{R, S}$ by property $\rm(s_2)$ for the pair $M, S$. In view of (c), this yields $r\in N_{R, \widetilde{S}}$.
 Therefore, property $\rm(s_2)$ holds for the pair $\widetilde{M}, \widetilde S$. This completes the proof of (b).
\end{proof}

\subsection{\it The Hesse pencil}\label{Hp} We intend to construct a relative section for the action of $PG$ on $X$. Our construction is based on consideration of the classical Hesse pencil
of cubics in ${\mathbb P}(V)$.

We consider in $U$ the $2$-dimensional linear subspace
\begin{equation*}\label{sbsL}
L:=\{\alpha(x_0^3+x_1^3+x_2^3\}+\beta x_0x_1x_2\mid (\alpha, \beta)\in \mathbb C^2
\},
\end{equation*}
and in ${\mathbb P}(U)$ we consider the line
\begin{equation}\label{ell}
\ell:=p_U(L\setminus \{0\})={\mathbb P}(L).
C^2, (\alpha, \beta)\neq (0, 0)\}\!\big),
\end{equation}
For every nonzero form $f=\alpha(x_0^3+x_1^3+x_2^3\}+\beta x_0x_1x_2\in L$, we put (see \eqref{cubic})
\begin{equation*}
\ell_{\beta/\alpha}:=p_U(f)\in \ell,\quad C_{\beta/\alpha}:=C(f),
\end{equation*}
where $\beta/0:=\infty$.

The {\it Hesse pencil} $\mathcal H$ and the {\it Hessian group} $\rm Hes$ are defined (cf.\;\cite{AD})  as respectively the set of plane cubics
\begin{equation*}\label{Hepe}
\mathcal H:=\{C_{\lambda}\mid \lambda\in \mathbb C\cup \infty\}
\end{equation*}
and its $PG$-normalizer (see \eqref{NZ})
\begin{equation}\label{fHes}
  {\rm Hes}:=N_{PG, \mathcal H}.
\end{equation}

The following Lemma \ref{proC} summarizes some of their properties that we need. Recall that $\varepsilon $ is a primitive cubic root of $1$.

\begin{lemma}\label{proC}
\
 \begin{enumerate}[\hskip 3mm\rm(i)]
 \item[$({\rm H}{}_1)$]\label{H1} For every elliptic curve $C$ in ${\mathbb P}(V)$, there is $g\in PG$ such that $g\cdot C\in \mathcal H$.
   \item[$({\rm H}{}_2)$]\label{H2} A cubic $C_\lambda$ is an elliptic curve
      if only if $\lambda\neq \infty, -3, -3\varepsilon, -3\varepsilon^2$.
   \item[$({\rm H}{}_3)$]\label{H3} For every $\lambda=\infty, -3, -3\varepsilon, -3\varepsilon^2$, the cubic $C_\lambda$
    is a triangle, i.e., a union of three different lines
  \begin{equation*}
  l_{\lambda, 1},  l_{\lambda, 2},  l_{\lambda, 3},
  \end{equation*}
  in ${\mathbb P}(V)$ with empty intersection,
       making the set $\mathcal L$ of $12$ lines
    altogether.\footnote{In \cite[Bemerkung, pp. 43--44]{Kra} it is incorrectly stated that if $C_\lambda$ for $\lambda\neq \infty$ is not smooth, then $\lambda=-3$ and $C_{-3}$ is a nodal cubic.}
   \item[$({\rm H}{}_4)$]\label{H4} A cubic $C$ in ${\mathbb P}(V)$ contains $\mathcal F$
    if and only if $C\in \mathcal H$.
   \item[$({\rm H}{}_5)$]\label{H5} If
   $C_\lambda$ is an elliptic curve,
      then $\mathcal F$ is the set ${\rm Fl}(C_{\lambda})$ of all its flexes.
          \item[$({\rm H}{}_6)$]\label{H6}
        On every line from $\mathcal L$ lie exactly three points of $\mathcal F$, and through every point of $\mathcal F$ pass exactly four lines from $\mathcal L$.
         \item[$({\rm H}{}_7)$]\label{H7} If a line $l$ in ${\mathbb P}(V)$ contains two different points of $\mathcal F$, then $l\in \mathcal L$.
             \item[$({\rm H}{}_8)$]\label{H8}  The group $Z_{PG, \mathcal F}$ \textup(see {\rm \eqref{NZ})} is trivial.
                               \item[$({\rm H}{}_9)$]\label{H9}
              $N_{PG, \mathcal F}= N_{PG, \mathcal L}=N_{PG, \ell}={\rm Hes}$.
  \item[$({\rm H}{}_{10})$]\label{H10}
    The bijection
  \begin{equation}\label{phi}
  \varphi\colon \mathcal F\to \mathbb F_3^2,\;t_{\boldsymbol{i},\boldsymbol{j}}\mapsto ({
  \boldsymbol
  i}, {
  \boldsymbol
  j}),
  \end{equation}
 preserves collinearity of points, i.e.,
   three different points
  $a, b, c\in \mathcal F$ are collinear in ${\mathbb P}(V)$ if and only if the points
  $\varphi(a)$, $\varphi(b)$, $\varphi(c)$ are collinear in
    the
  two-dimensional
  affine space  $\mathbb F_3^2$ over $\mathbb F_3$.
    \item[$({\rm H}{}_{11})$]\label{H11}
        If $g\in N_{PG, \mathcal F}$, then  $\varphi\circ g\circ \varphi^{-1}\in {\rm SAff}(\mathbb F_3^2)$
            and the map
      \begin{equation*}
      N_{PG, \mathcal F}\to {\rm SAff}(\mathbb F_3^2),\;g\mapsto \varphi\circ g\circ \varphi^{-1}
        \end{equation*}
        is a group isomorphism.
        \item[$({\rm H}{}_{12})$]\label{H12}  If $g\in PG$ and $C\in\mathcal H$ is an elliptic curve, then $g\cdot C\in\mathcal H$ if and only if $g\in {\rm Hes}$.
 \end{enumerate}
\end{lemma}

\begin{proof}

$({\rm H}{}_{1})$ See \cite[pp.\,291--293, Thm.\,4]{BK}.

\vskip 1mm

$({\rm H}{}_{2})$, $({\rm H}{}_{3})$, $({\rm H}{}_{4})$, $({\rm H}{}_{5})$ See \cite[pp.\,294--295, Prop.\,5]{BK}.

\vskip 1mm

$({\rm H}{}_{6})$  If $\lambda=\infty, -3, -3\varepsilon, -3\varepsilon^2$, then $f$ in $C_{\lambda}=C(f)$ is a product of three linear forms whose zero sets in ${\mathbb P}(V)$ are the lines $l_{\lambda, 1},  l_{\lambda, 2}, l_{\lambda, 3}$.
In the explicit form this product looks as follows  (see \cite[p.\,294]{BK}):
\begin{equation}\label{liness}
\begin{split}
C_\infty&\!=C(x_0x_1x_2),\\
C_{-3}&\!=C\big((x_0\!+\!x_1\!+\!x_2)(x_0\!+\!\varepsilon x_1\!+\! \varepsilon^2 x_2)(x_0\!+\!\varepsilon^2 x_1\!+\! \varepsilon x_2)\big),\\
C_{-3\varepsilon}&\!=\!C\big((x_0\!+\!x_1\!+\!\varepsilon x_2)(x_0\!+\!\varepsilon x_1\!+\!  x_2)(x_0\!+\!\varepsilon^2 x_1\!+\! \varepsilon^2 x_2)\big),\\
C_{-3\varepsilon^2}&\!=C\big((x_0\!+\!x_1\!+\!\varepsilon^2 x_2)(x_0\!+\!\varepsilon x_1\!+\! \varepsilon x_2)(x_0\!+\!\varepsilon^2 x_1\!+\!  x_2)\big).
\end{split}
\end{equation}
From \eqref{pij} and \eqref{liness} it follows directly that the relative position of the points from $\mathcal F$  and the lines from $\mathcal L$ is as
depicted in the following Figure\,1:
\begin{figure}[htbp]\center
\includegraphics[width=90mm]{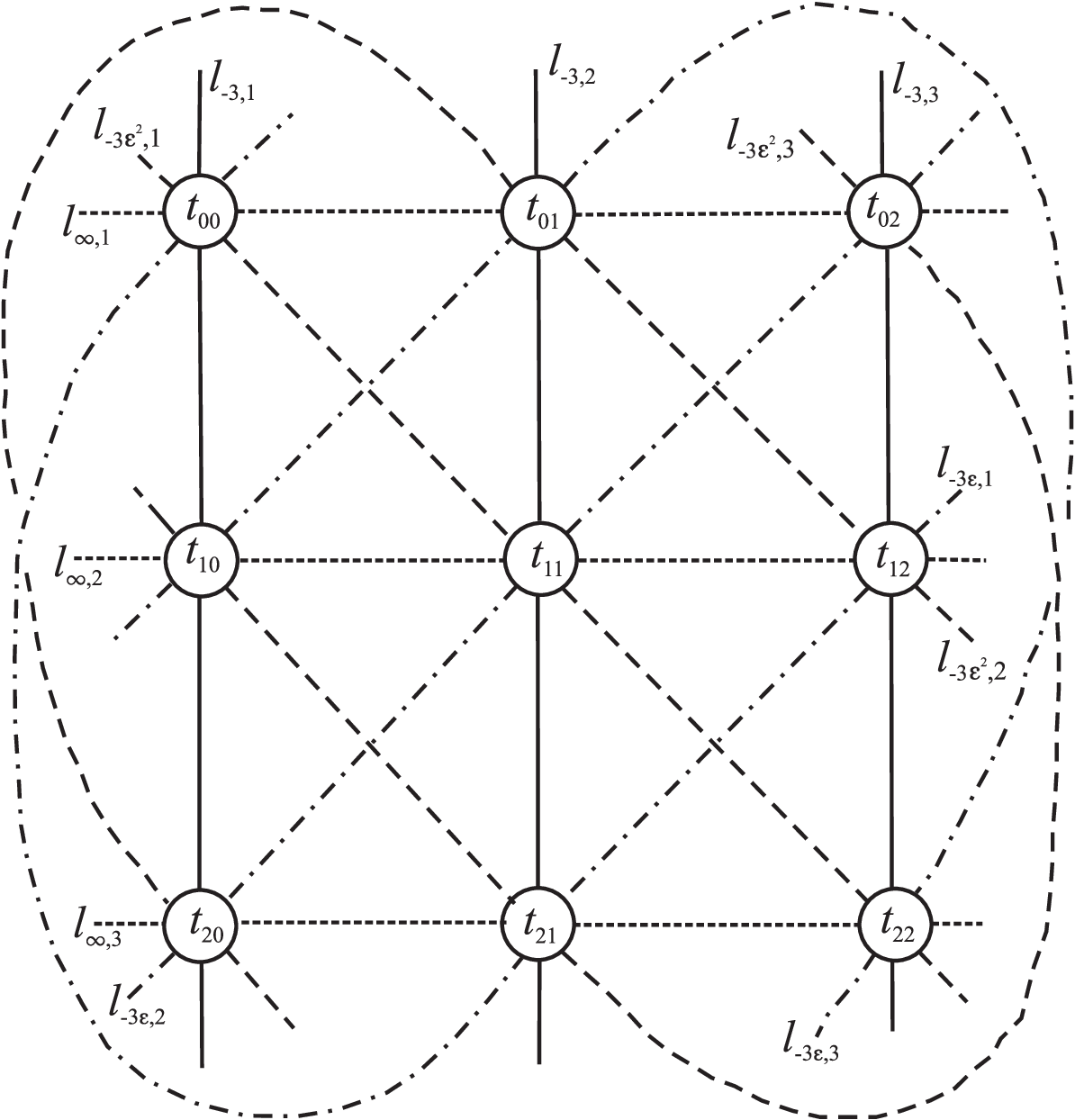}
\end{figure}
\vskip 6mm
\centerline{\text{\sc
Figure\,1}}

\vskip 4mm

\noindent This, in turn, clearly implies $({\rm H}{}_{6})$.

\vskip 1mm

$({\rm H}{}_{7})$ This follows from the fact that for every two different points of $\mathcal F$, there is a line from $\mathcal L$ containing them (see Figure 1).

\vskip 1mm

$({\rm H}{}_{8})$ In view of the first fundamental theorem of projective geometry
\cite[4.5.10]{Ber},
this follows from the fact that $\mathcal F$ contains a projective frame of ${\mathbb P}(V)$, for instance,  $t_{\boldsymbol{0},\boldsymbol{0}}$, $t_{\boldsymbol{1},\boldsymbol{0}}$, $t_{\boldsymbol{1},\boldsymbol{1}}$, $t_{\boldsymbol{2},\boldsymbol{1}}$, as is readily seen from Figure\,1.

\vskip 1mm

$({\rm H}{}_{9})$  If $g\in N_{PG, \ell}$ and $C\in \mathcal H$, i.e., $C=C(f)$ for $f\in \ell$, then
$g\cdot C=g\cdot C(f)=C(g^{-1}\cdot f)\in\mathcal H$. Therefore, $g\in N_{PG, \mathcal H}$. Conversely, if $g\in N_{PG, \mathcal H}$ and $f\in \ell$, then $C=C(f)\in \mathcal H$, therefore,  $g\cdot C=C(g^{-1}\cdot f)\in\mathcal H$, hence $g^{-1}\cdot f\in \ell$. Therefore, $g\in N_{PG, \ell}$. This proves that $N_{PG, \ell}=N_{PG, \mathcal H}=:{\rm Hes}$.

 If $g\in H_{PG, \mathcal H}$ and $C\in \mathcal H$ is an elliptic curve, then $g\cdot C\in \mathcal H$ and $\mathcal F={\rm Fl}(g\cdot C)=g\cdot {\rm Fl}(C)=g\cdot {\mathcal F}$
 in view of $({\rm H}{}_{5})$. Therefore, $g\in N_{PG, \mathcal F}$.
 Conversely, let $g\in N_{PG, \mathcal F}$ and $C\in\mathcal H$. By   $({\rm H}{}_{4})$, the latter and former inclusions imply respectively the inclusions $\mathcal F\subset C$ and
 $\mathcal F\subset g\cdot C$. Therefore, $g\in N_{PG, \mathcal H}$. This proves
 $N_{PG, \mathcal F}=N_{PG, \mathcal H}$.

Finally, if $g\in N_{PG, \mathcal F}=N_{PG, \mathcal H}$ and $C\in \mathcal H$ is a singular cubic  (triangle), then the cubic $g\cdot C\subset \mathcal H$ is singular as well. This and the definition of $\mathcal L$
(see   $({\rm H}{}_{3})$) entail that $g\cdot \mathcal L=\mathcal L$, i.e., $g\in N_{PG, \mathcal L}$. Conversely, if $g\in N_{PG, \mathcal L}$, i.e., $g\cdot \mathcal L$, then
$g\cdot \mathcal F=\mathcal F$ because $\mathcal F$  can be geometrically characterized as the set of all such points of ${\mathbb P}(V)$ through which four different lines from $\mathcal L$ pass. Therefore, $g\in N_{PG, \mathcal F}$. This proves that
 $N_{PG, \mathcal L}=N_{PG, \mathcal F}$.

\vskip 1mm

$({\rm H}{}_{10})$
Choosing $(\boldsymbol{0}, \boldsymbol{0})$ as the origin, we identify the affine space $\mathbb F_3^2$ with the vector space of its parallel translations. The unique line in $\mathbb F_3^2$ passing through two different points $u, v$
is the set $\{u+t(v-u)\mid t\in \mathbb F_3\}$. It consists of three points $u$, $v$, and $u+\boldsymbol{2}(v-u)=2(u+v)$. Hence there are $\tfrac{1}{3}\binom{9}{2}=12$ lines in $\mathbb F_3^2$ altogether.

Fix an elliptic curve $C_\lambda\in \mathcal H$. By $({\rm H}{}_{5})$,  the set of all its flexes is $\mathcal F$. Applying the classical construction \cite[III, 3.2]{Sha}, we endow $C_\lambda$ with the structure of a one-dimensional Abelian variety  $A$ taking a point $o\in \mathcal F$ as the zero of the group law $\oplus$. Then $\mathcal F$ is the sub\-group of $A$ of order $9$ and $\mathcal F\setminus \{o\}$ is the set of
all elements of $A$ of or\-der $3$. Hence this subgroup is isomorphic to the vector group of $\mathbb F_3^2$. Every group isomorphism $\alpha\colon \mathcal F\to \mathbb F_3^2$ preserves collinearity. Indeed, by \cite[p.\,175]{Sha}, collinearity in ${\mathbb P}(V)$ of three different points $a, b, c\in \mathcal F$ is equivalent to
\
\vskip -4mm
\begin{equation}\label{=}
a\oplus b\oplus c=o.
\end{equation}
Since $\alpha$ is a group isomorphism,  \eqref{=} is equivalent to
$\alpha(a)+\alpha(b)+\alpha(c)=(\boldsymbol{0}, \boldsymbol{0})$, i.e.,
$\alpha(c)=2(\alpha(a)+\alpha(b))$. As explained above, the latter means collinearity of
$\alpha(a)$, $\alpha(b)$, $\alpha(c)$ in $\mathbb F_3^2$.

The aforesaid implies that $a\oplus b$ for any $a, b\in\mathcal F$ is obtained as follows. If $a\neq b$, let $c\in \mathcal F$ be the unique point such that $a$, $b$, $c$ are collinear. If $c \neq o$, then $a\oplus b$ is the unique point $d\in \mathcal F$ such that $c, d, o$ are collinear; if $c=o$, then $a\oplus b=o$.
If $a=b\neq o$, then $a\oplus b$ is the unique point $e\in \mathcal F$ such that $a, e, o$ are collinear.

Using this description and Figure 1, it is directly verified that \eqref{phi}  is a group isomorphism for $o=(\boldsymbol{0}, \boldsymbol{0})$, hence it preserves collinearity.

\vskip 1mm

$({\rm H}{}_{11})$  Since $g$ preserves collinearity of points of ${\mathbb P}(V)$, from $({\rm H}{}_{10})$ we infer that  $\varphi\circ g\circ \varphi^{-1}$  is the permutation of $\mathbb F_3^2$ preserving collinearity of points of $\mathbb F_3^2$. Hence, by the fundamental theorem of affine geometry \cite[2.6.3]{Ber}, this permutation is semiaffine, and therefore, affine because
the automorphism group of
$\mathbb F_3$ is trivial (see \cite[V, Thm.\,5.4]{Lan}).
Thus we obtain the homomorphism
\begin{equation*}\label{iota}
\theta\colon N_{PG, \mathcal F}\to {\rm Aff}(\mathbb F_3^2),\;g\mapsto \varphi\circ g\circ \varphi^{-1}.
\end{equation*}
In view of $({\rm H}{}_{8})$, its kernel is trivial. Therefore, it remains to show that  its image is
${\rm SAff}(\mathbb F_3^2)$.

First, ${\rm SAff}(\mathbb F_3^2)\subseteq \theta(N_{PG, \mathcal F})$. To see this, it suffices to present a subset of $N_{PG, \mathcal F}$ whose image under $\theta$ generates the group ${\rm SAff}(\mathbb F_3^2)$. We claim that the subset of five elements $g_1, g_2, g_3, g_4, g_5\in {\mathbb P}(V)$ defined by the second column of Table 3 below shares this property:

\vskip 2mm
\centerline{\sc Table $3$}

\vskip 3mm
\begin{center}
\begin{tabular}{c||c|c}
$g$&$g\cdot (\alpha_0:\alpha_1:\alpha_2)\;\mbox{for every $(\alpha_0:\alpha_1:\alpha_2)\in {\mathbb P}(V)$}$&$g\cdot t_{\boldsymbol{i},\boldsymbol{j}}$
\\[2pt]
\hline
 \hline
 $g_1$&$(\alpha_1:\alpha_2:\alpha_0)$&$t_{{\boldsymbol i}+{\boldsymbol 2},{\boldsymbol j}}$\\
 \hline
  $g_2$&$(\alpha_0:\alpha_2:\alpha_1)$&$t_{{\boldsymbol {2i}},{\boldsymbol {2j}}}$\\
 \hline
  $g_3$&$(\alpha_0:\varepsilon\alpha_1:\varepsilon^2\alpha_2)$&$t_{{\boldsymbol i},{\boldsymbol j}+\boldsymbol{2}}$\\
 \hline
  $g_4$&$(\alpha_0:\varepsilon\alpha_1:\varepsilon\alpha_2)$&$t_{{\boldsymbol i},{\boldsymbol{i}}+{\boldsymbol{ j}}}$\\
 \hline
  $g_5$&$(\alpha_0+\alpha_1+\alpha_2:\alpha_0+\varepsilon\alpha_1
  +\varepsilon^2\alpha_2:\alpha_0+\varepsilon^2\alpha_1+\varepsilon\alpha_2)$&$t_{{\boldsymbol j},{\boldsymbol{2}}{\boldsymbol i}}$\\
 \end{tabular}
 \end{center}
 \vskip 3mm
\noindent  Indeed, using this definition and  \eqref{pij} it is directly verified that every $g_k\cdot t_{\boldsymbol{i},\boldsymbol{j}}$ is as specified in the third column of Table 3. It follow from this column that
\begin{enumerate}[\hskip 4.2mm\rm(a)]
\item $g_k\in N_{{\mathbb P}(V),\mathcal F}$ for every $k$;
\item $\theta(g_1)$ and $\theta(g_3)$ are the parallel translations of $\mathbb F_3^2$ by the vectors $(\boldsymbol{2}, \boldsymbol{0})$ and $(\boldsymbol{2}, \boldsymbol{0})$,  respectively;
    \item $\theta(g_2)$, $\theta(g_4)$, and  $\theta(g_5)$ are the linear transformations of $\mathbb F_3^2$ whose matrices in the basis $(\boldsymbol{1},\boldsymbol{0})$, $(\boldsymbol{0},\boldsymbol{1})$ are respectively $\big(\begin{smallmatrix} {\boldsymbol{2}}&{\boldsymbol{0}}\\
         {\boldsymbol{0}}&{\boldsymbol{2}}\end{smallmatrix}\big)$,
         $ \big(\begin{smallmatrix}
        {\boldsymbol{1}}&{\boldsymbol{0}}\\
         {\boldsymbol{1}}&{\boldsymbol{1}}
        \end{smallmatrix}\big)$, and
        $\big(
        \begin{smallmatrix}
        {\boldsymbol{0}}&{\boldsymbol{1}}\\
         {\boldsymbol{2}}&{\boldsymbol{0}}
        \end{smallmatrix}
        \big)$.
           \end{enumerate}

    \noindent As the vectors specified in (b) generate the vector group of $\mathbb F_3^2$, and the matrices specified in (c) generate the group ${\rm {SL}}_2(\mathbb F_3)$, this proves
    the claim maid.

   Thus ${\rm SAff}(\mathbb F_3^2)\subseteq \theta(N_{PG, \mathcal F})$. To prove the equality, it suffices, in view of
        $[{\rm Aff}(\mathbb F_3^2): {\rm SAff}(\mathbb F_3^2)]=2$, to show that
    ${\rm Aff}(\mathbb F_3^2)\neq \theta(N_{PG, \mathcal F})$. The map
        \begin{equation}\label{thetaaa}
    c\colon {\mathbb P}(V)\to {\mathbb P}(V), \quad (\alpha_0:\alpha_1:\alpha_2)\mapsto (\overline{\alpha_0}:\overline{\alpha_1}:\overline{\alpha_2}),
    \end{equation}
        where $\overline{\alpha_i}$ is complex conjugate of $\alpha_i$,
      is a bijection that preserves collinearity of points. It follows from \eqref{pij}, \eqref{thetaaa} that
    \begin{equation}\label{ttheta}
    \left.\begin{split}
    c(t_{{\boldsymbol 0}, {\boldsymbol 0}})&=t_{{\boldsymbol 0}, {\boldsymbol 0}},\quad
     c(t_{{\boldsymbol 0}, {\boldsymbol 1}})=t_{{\boldsymbol 0}, {\boldsymbol 2}},\quad
      c(t_{{\boldsymbol 0}, {\boldsymbol 2}})=t_{{\boldsymbol 0}, {\boldsymbol 1}},\\
       c(t_{{\boldsymbol 1}, {\boldsymbol 0}})&=t_{{\boldsymbol 1}, {\boldsymbol 0}},\quad
        c(t_{{\boldsymbol 1}, {\boldsymbol 1}})=t_{{\boldsymbol 1}, {\boldsymbol 2}},\quad
         c(t_{{\boldsymbol 1}, {\boldsymbol 2}})=t_{{\boldsymbol 1}, {\boldsymbol 1}},\\
          c(t_{{\boldsymbol 2}, {\boldsymbol 0}})&=t_{{\boldsymbol 2}, {\boldsymbol 0}},\quad
           c(t_{{\boldsymbol 2}, {\boldsymbol 1}})=t_{{\boldsymbol 2}, {\boldsymbol 2}},\quad
            c(t_{{\boldsymbol 2}, {\boldsymbol 2}})=t_{{\boldsymbol 2}, {\boldsymbol 1}},
    \end{split}
    \right\}
    \end{equation}
so    $c(\mathcal F)=\mathcal F$. In view of $({\rm H}{}_{10})$, this yields the bijection
\begin{equation*}\label{ftheta}
\widehat{c}:=\varphi\circ c\circ\varphi^{-1}\colon\mathbb F_3^2\to \mathbb F_3^2
\end{equation*}
that preserves collinearity of points; as above, the latter implies $\widehat{c}\in {\rm Aff}(\mathbb F_3^2)$.
We claim that $\widehat{c}\notin \theta(N_{{\mathbb P}(V),\mathcal F})$.
Arguing by contradiction, we assume the contrary, i.e., there is $g\in N_{PG, \mathcal F}$ such that \begin{equation}\label{ggg}
g\cdot t_{\boldsymbol{i}, \boldsymbol{j}}=c(t_{\boldsymbol{i}, \boldsymbol{j}})\;\,\mbox{for all
$\boldsymbol{i}, \boldsymbol{j}$.}
\end{equation}
It follows from
Figure 1, \eqref{ggg}, and \eqref{ttheta} that $t_{{\boldsymbol 2},  {\boldsymbol 1}}$,  $t_{{\boldsymbol 0},  {\boldsymbol 0}}$,
$t_{{\boldsymbol 1},  {\boldsymbol 0}}$, $t_{{\boldsymbol 1},  {\boldsymbol 1}}$  and
$t_{{\boldsymbol 2},  {\boldsymbol 2}}$,  $t_{{\boldsymbol 0},  {\boldsymbol 0}}$,
$t_{{\boldsymbol 1},  {\boldsymbol 0}}$, $t_{{\boldsymbol 1},  {\boldsymbol 2}}$ are two frames of ${\mathbb P}(V)$ such that $g$ maps the first one to the second. Therefore, the homogeneous coordinates of  $t_{{\boldsymbol 0},  {\boldsymbol 1}}$ in the first frame should coincide with the homogeneous coordinates of $g\cdot t_{{\boldsymbol 0},  {\boldsymbol 1}}=t_{{\boldsymbol 0},  {\boldsymbol 2}}$ in the second. But a direct calculation of these coordinates shows that they do not actually coincide. This contradiction completes the proof of $({\rm H}{}_{11})$.

\vskip 1mm

$({\rm H}{}_{12})$  By $({\rm H}{}_{5})$,
we have
${\rm Fl}(C)=\mathcal F$. Therefore, by \eqref{gFl},
\begin{equation}\label{FlFl}
{\rm Fl}(g\cdot C)=g\cdot \mathcal F.
\end{equation}
If $g\cdot C\in \mathcal H$, then $({\rm H}{}_{5})$ entails ${\rm Fl}(g\cdot C)=\mathcal F$, so \eqref{FlFl} yields
\begin{equation}\label{gFF}
g\cdot \mathcal F=\mathcal F,
\end{equation}
\vskip -2mm
\noindent  i.e., $g\in N_{PG, \mathcal F}\overset{({\rm H}{}_9)}{=}{\rm Hes}$. Conversely, if  $g\in N_{PG, \mathcal F}$, then \eqref{gFF} holds, therefore, \eqref{FlFl} yields ${\rm Fl}(g\cdot C)=\mathcal F$. In view of
$({\rm H}{}_{4})$, the latter equality  implies $g\cdot C\in \mathcal H$.
\end{proof}

\begin{rem} The reviewer of this paper asked about
the part of the proof of $({\rm H}{}_{11})$ that uses formula \eqref{thetaaa}:
``Is a result of this form valid
only over the field of complex numbers, or something similar can be done over
any algebraically closed field of characteristic zero?''
The answer is that both the result and essentially its proof hold true over
any algebraically closed field $k$ of characteristic zero. Indeed, let $\overline{\mathbb Q}$ be the algebraic closure of $\mathbb Q$ in $\mathbb C$ and let $
\kappa\in {\rm Aut}_{\mathbb Q}(\overline{\mathbb Q})$ be the restriction to $\overline{\mathbb Q}$ of the complex conjugation of $\mathbb C$. The assumptions on $k$
permit to identify $\overline{\mathbb Q}$ with a subfield of $k$
and extend $\kappa$ to a field automorphism $k\to k$, $\alpha\mapsto \widetilde{\alpha}$ (see \cite[Chap.\,V, Cor.\,2.9]{Lan}). Since
every $t_{{\boldsymbol i}, {\boldsymbol j}}$ is a $\overline{\mathbb Q}$-rational point of ${\mathbb P}(V)$,  when replacing $\mathbb C$ with $k$, the above proof of $({\rm H}{}_{11})$ remains unchanged if in \eqref{thetaaa} every $\overline{\alpha_i}$ is replaced with $\widetilde{\alpha_i}$.
\end{rem}

\subsection{\it Relative sections for the actions of $PG$ on ${\mathbb P}(U)$ and $X$}\label{Hp}

\begin{theorem}\label{PU}
The line $\ell$ \textup(see {\rm \eqref{ell})} is a relative section for the action of $PG$ on ${\mathbb P}(U)$. Its normalizer in $PG$ is the Hessian group ${\rm Hes}$.
\end{theorem}

\begin{proof} For the triple $M={\mathbb P}(U)$, $S=\ell$,
$R=PG$,
property ${\rm(s_1)}$ in De\-fi\-ni\-tion \ref{relsec} holds because
of Lemmas \ref{S}\eqref{Sb}, \ref{proC}${\rm (H_1)}$, and irreducibility of
$\ell$.
Property ${\rm(s_2)}$ holds because of Lemma
\ref{proC}${\rm (H_2)}$,${\rm (H_9)}$,${\rm (H_{12})}$. This proves the first statement. The second follows from
Lemma \ref{proC}${\rm (H_9)}$.
\end{proof}

\subsection{\it Proof of Theorem {\rm\ref{main2}}} \label{pr2}
We take a nonzero form $f\in L$ and put $a=p_U(f)\in\ell$. It follows from
Lem\-ma\;\ref{proC}({\rm H}${}_{2}$),({\rm H}${}_{5}$) that
\begin{equation}\label{fiber1}
\pi_9^{-1}(a)=
\bigcup_{\boldsymbol{i},\boldsymbol{j}=\boldsymbol{0}}^{\boldsymbol{2}}\big(a, t_{\boldsymbol{i},\boldsymbol{j}}\big)\quad \mbox{if $a\neq \ell_\infty, \ell_{-3}, \ell_{-3\varepsilon}, \ell_{-3\varepsilon^2}$}.
\end{equation}
By Lem\-ma\;\ref{proC}({\rm H}${}_{3}$),
if  $\lambda=\infty$,
$-3$, $-3\varepsilon$, $-3\varepsilon^2$ and $a=\ell_\lambda$, then $C(f)$ is the triangle
$l_{\lambda, 1}\cup l_{\lambda, 2}\cup l_{\lambda, 3}$. By Lemma \ref{class}, this entails $h_{\mu, 6}\in G\cdot f$  for some $\mu$. In view of \eqref{gFl} and Table 2, this, in turn, yields
\begin{equation}\label{fiber2}
\pi_9^{-1}(a)=
\bigcup_{k=1}^3\big(a\times l_{\lambda, k}\big)\quad \mbox{if $a=\ell_{\lambda}$ for $\lambda=\infty, -3, -3\varepsilon, -3\varepsilon^2$.}
\end{equation}

From \eqref{fiber1}, \eqref{fiber2} we infer that every irreducible component of $\pi_9^{-1}(\ell)$ is a line, there are exactly $9+12=21$ such irreducible components, and those of them which $\pi_9$ dominantly maps to $\ell$ are the lines $\ell\times t_{\boldsymbol{l},\boldsymbol{k}}$, where $(\boldsymbol{l},\boldsymbol{k})\in\mathbb F_3^2$.

In view of Theorem \ref{PU}, Lemma \ref{55}(b),(c), and Lemma \ref{rese}(b),(c), each
line
$\ell\times t_{\boldsymbol{i},\boldsymbol{j}}$ is a relative section for the action of $PG$ on $X$, and $N_{PG, \ell\times t_{\boldsymbol{i},\boldsymbol{j}}}=N_{N_{PG, \ell}, \ell\times t_{\boldsymbol{i},\boldsymbol{j}}}$. By Theorem \ref{PU}, we have $N_{PG, \ell}={\rm Hes}$. Hence
$N_{PG, \ell\times t_{\boldsymbol{i},\boldsymbol{j}}}$ is the ${\rm Hes}$-stabilizer ${\rm Hes}_{\boldsymbol{i},\boldsymbol{j}}$
of $t_{\boldsymbol{i},\boldsymbol{j}}$.
Clearly,
the ${\rm SAff}(\mathbb F_3^2)$-stabilizer of any point of $\mathbb F_3^2$ is isomorphic to ${\rm SL}_2(\mathbb F_3)$. In view of Lemma \ref{proC}$({\rm H}_{10}), ({\rm H}_{11})$, this means that ${\rm Hes}_{\boldsymbol{i},\boldsymbol{j}}$ is isomorphic to ${\rm SL}_2(\mathbb F_3)$.
Finally, the isomor\-phism
$$
\ell\to \ell\times t_{\boldsymbol{i},\boldsymbol{j}},\;a\mapsto (a, t_{\boldsymbol{i},\boldsymbol{j}})
$$
is clearly ${\rm Hes}_{\boldsymbol{i},\boldsymbol{j}}$-equivariant. In view of
Lemma \ref{rese}(e), this proves (a).

The inefficiency kernel of the $G$-action on $U$ is the center $Z$ of $G$. Therefore, this action determines the faithful action of $PG$ on $U$. Clearly, $N_{PG, L}=N_{PG, \ell}={\rm Hes}$, so the subgroup ${\rm Hes}_{\boldsymbol{i},\boldsymbol{j}}$ of ${\rm Hes}$ linearly acts of $L$. This action determines $PG\times ^{{\rm Hes}_{\boldsymbol{i},\boldsymbol{j}}} L$, the rank $2$ homogeneous vector bundle over $PG/{\rm Hes}_{\boldsymbol{i},\boldsymbol{j}}$ with fiber $L$. Clearly, its projectivization is $PG\times ^{{\rm Hes}_{\boldsymbol{i},\boldsymbol{j}}} \ell$. This proves (b).\hfill $\Box$

\section{Rationality of $X$}

\subsection{\it Rationality of some homogeneous spaces} In this subsection,
ratio\-nality of some homogeneous spaces is proved. It will be used in the proof of Theorem \ref{main3} given in Subsection \ref{pr3}.

\begin{theorem}\label{ratY} For every finite subgroup $K$ of the group $G=\s3$, the $8$-dimensional homogeneous space $G/K$
 is a rational algebraic variety.
\end{theorem}
\begin{proof} Let $\nu\colon G\to PG={\rm PSL}_3(\mathbb C)$
be the canonical projection. Its kernel is the center of $G$:
$$
{\rm ker}\, \nu=Z=\{{\rm diag}(\varepsilon, \varepsilon, \varepsilon)\mid \varepsilon \in\mathbb C^*, \varepsilon^3=1\}.
$$
 As the order
  of $Z$ is prime number, one of the following
possibi\-lities
holds:
\begin{enumerate}[\hskip 4.2mm\rm(a)]
\item $K\bigcap Z$ is trivial,
\item $Z\subseteq K$.
\end{enumerate}
We shall explore them
separately.

\vskip 1mm

{\it Case} (a).
Let (a) holds. We consider in $G$ the $6$-dimensional parabolic subgroup
\begin{equation*}
P:= \bigg\{\hskip -1.5mm\begin{pmatrix}
A&\ast\\
0&\det A^{-1}
\end{pmatrix}\,\Big\vert\, A\in{\rm GL}_2(\mathbb C)\hskip -.6mm\bigg\}.
\end{equation*}
All maximal connected semisimple subgroups of $P$ are isomorphic to ${\rm SL}_2(\mathbb C)$. By \cite[Thm. 2.8]{PV}, this entails that $P$ is a special group in the sense of Serre \cite[4.1]{Ser}.

We claim that the natural action of $P$ on $G/K$ is generically free, i.e., $P$-stabilizers of points in general position in
$G/K$ are trivial.

To prove this claim, note that if $\pi\colon G\to G/K$ is the canonical projection, then for every $g\in G$, the $P$-stabilizer of $\pi(g)$ is the finite group
\begin{equation}\label{stabil}
P_{\pi(g)}:=P\cap gKg^{-1}.
\end{equation}

Let $K_1,\ldots, K_d$ be all nontrivial subgroups of the finite group $K$. By
\cite[Chap.\,I, 1.7, Prop.]{Bor}, every set
\begin{equation}\label{ili}
M_i:=\{g\in G\mid gK_ig^{-1}\subseteq P\}\overset{\rm \eqref{stabil}}{=}\{g\in G\mid gK_ig^{-1}\subseteq P_{\pi(g)}\}
\end{equation}
is closed in $G$. Therefore,
\begin{equation}
\label{cL}
M:=M_1\cup\ldots\cup M_d
\end{equation}
is closed in $G$ as well.

If $g\in M$, then $g\in M_i$ for some $i$, hence $P_{\pi(g)}$ is nontrivial in view of \eqref{ili}.
Conversely, if $P_{\pi(g)}$ is nontrivial for some $g\in G$, then in view of \eqref{stabil}, there is $i$ such that $P_{\pi(g)}=gF_ig^{-1}$; whence $g\in M_i$ by \eqref{ili}.

Thus,
$P_{\pi(g)}$ is trivial if and only if $g$ lies in the open subset $G\setminus M$ of $G$. This subset in nonempty. Indeed, arguing by contradiction, we assume the contrary, i.e., $G=M$. Since the variety $G$ is irreducible, \eqref{cL} then implies that $G=M_i$ for some $i$. In view of \eqref{ili}, this yields
\begin{equation}\label{interse}
K_i\subseteq \bigcap_{g\in G}gPg^{-1}.
\end{equation}
The right-hand side of \eqref{interse} is a proper normal algebraic subgroup of $G$. Since $G$ is simple, this means that this subgroup lies in $Z$. Therefore, $K_i\subseteq Z$ by \eqref{interse}.  This contradicts the assumption that (a) holds, because $K_i$ is a nontrivial subgroup of $K$.

Thus, $G\setminus M\neq \varnothing$, and therefore, $\pi(G\setminus M)$ contains a nonempty open subset of $G/K$. Since $P_a$ is trivial for every $a\in \pi(G\setminus M)$, this proves the claim that the $P$-action on $G/K$ is generically free.

  In view of \cite[\S4]{Ser}, \cite[Sect. 2.5 and Thm. 2.8]{PV}, since $P$ is special, this claim  implies that $G/K$ is birationally isomorphic to the pro\-duct of $P$ and
   an irreducible variety
   $B$ whose
field of rational functions
is $\mathbb C(G/K)^P$. It follows from rationality of the underlying variety of every connected affine algebraic group \cite{C}
that $P$ is rational and
$B$ is unira\-tio\-nal. On the other hand,
$\dim B=
\dim G/K-\dim P=
8-6=2$.
Therefore, by Castelnuovo's theorem, $B$ is rational. Thus $G/K$ is birationally isomorphic to a product of rational varieties, hence itself is rational.
This completes the proof that $G/K$ is rational if (a) holds.

\vskip 2mm

{\it Case} (b). Let (b) holds. Then $G/K$ is isomorphic to ${\rm PSL}_3(\mathbb C)/\nu(K)$. There\-fore, the proof will be complete if for every finite subgroup $D$ of ${\rm PSL}_3(\mathbb C)$, rationality of  ${\rm PSL}_3(\mathbb C)/D$ is proved.

We consider in ${\rm PSL}_3(\mathbb C)$ the parabolic sub\-group $Q:=\nu (P)$. The center of every maximal connected semi\-simple subgroup of $P$ has order $2$ and therefore, does not lie in ${\rm ker}\,\tau$. Hence, like for $P$, all maximal connected semi\-simple sub\-groups in $Q$ are isomorphic to ${\rm SL}_2(\mathbb C)$, there\-fore, $Q$ is a special group in the sense of Serre. As the  center of ${\rm PSL}_3(\mathbb C)$ is trivial, its intersection with $D$
is trivial.
Now the same reasoning applies as above when considering case (a), if we replace in it $G$ with ${\rm PSL}_3(\mathbb C)$, $P$ with $Q$, and $K$ with $D$. This completes the proof of Theorem \ref{ratY}.
\end{proof}

\begin{corollary}\label{c3} For every finite subgroup $D$ of the group  ${\rm PSL}_3(\mathbb C)$, the homogeneous space ${\rm PSL}_3(\mathbb C)/D$ is a rational al\-gebraic variety.
\end{corollary}

\subsection{\it Proof of Theorem {\rm\ref{main3}}}\label{pr3}  By Theorem \ref{main2}, $X$ is birationally isomorphic to
the homogeneous fiber space
$PG\times^{{\rm Hes}_{\boldsymbol{i},\boldsymbol{j}}}\ell$ and the latter is the projectivization of the homogeneous vector bundle $PG\times^{{\rm Hes}_{\boldsymbol{i},\boldsymbol{j}}} L$. Since $PG\times^{{\rm Hes}_{\boldsymbol{i},\boldsymbol{j}}} L$ is locally trivial in the Zariski topology,
$PG\times^{{\rm Hes}_{\boldsymbol{i},\boldsymbol{j}}}\ell$
is locally trivial in the Zariski topology, too. Therefore, $X$ is birationally isomorphic to
$(PG/{\rm Hes}_{\boldsymbol{i},\boldsymbol{j}})\times \ell$.  The homogeneous space $PG/{\rm Hes}_{\boldsymbol{i},\boldsymbol{j}}$ is rational by Corol\-la\-ry\;\ref{c3}. Since the line $\ell$ is rational, too, this implies rationality of $X$.\hfill $\Box$

\end{document}